\documentclass[reqno, 10pt]{article}
\usepackage{amsmath,amsfonts,amssymb,amsthm}

\usepackage{geometry}
\usepackage{bbm}
\newtheorem{theorem}{Theorem}
\newtheorem{lemma}[theorem]{Lemma}

\newtheorem{remark}[theorem]{Remark}

\theoremstyle{definition}

\newtheorem{example}[theorem]{Example}

\newcommand\eps{\varepsilon}
\renewcommand\le{\leqslant}
\renewcommand\ge{\geqslant}
\renewcommand\varrho{\rho}

\newcommand\RR{{\mathbb R}}

\newcommand\E{\operatorname{\mathbb E{}}}

\renewcommand\Pr{\operatorname{\mathbb P{}}}
\newcommand\Var{\operatorname{Var}}

\newcommand\Bin{\operatorname{Bin}}

\newcommand\floor[1]{\lfloor #1 \rfloor}
\newcommand\ceil[1]{\lceil #1 \rceil}

\newcommand\set[1]{\ensuremath{\{#1\}}}

\newcommand\Bigset[1]{\ensuremath{\Bigl\{#1\Bigr\}}}

\newcommand\bigpar[1]{\bigl(#1\bigr)}
\newcommand\Bigpar[1]{\Bigl(#1\Bigr)}
\newcommand\biggpar[1]{\biggl(#1\biggr)}

\newcommand\bigsqpar[1]{\bigl[#1\bigr]}
\newcommand\Bigsqpar[1]{\Bigl[#1\Bigr]}

\newcommand\abs[1]{|#1|}
\newcommand\bigabs[1]{\bigl|#1\bigr|}
\newcommand\Bigabs[1]{\Bigl|#1\Bigr|}

\newcommand\oY{{\overline Y}}
\newcommand\oy{{\overline y}}
\newcommand\ty{{\tilde{y}}}

\newcommand\cB{{\mathcal B}}

\newcommand\cD{{\mathcal D}}
\newcommand\cE{{\mathcal E}}
\newcommand\cF{{\mathcal F}}
\newcommand\cG{{\mathcal G}}

\newcommand\cM{{\mathcal M}}

\newcommand{\indic}[1]{\mathbbm{1}_{\{{#1}\}}}

\newcommand\dd{\,\mathrm{d}}

\let\OLDthebibliography\thebibliography
\renewcommand\thebibliography[1]{
  \OLDthebibliography{#1}
  \setlength{\parskip}{0pt}
  \setlength{\itemsep}{0pt plus 0.3ex}
}

\begin{document}
\title{On Wormald's differential equation method}
\author{Lutz Warnke%
\thanks{School of Mathematics, Georgia Institute of Technology, Atlanta GA~30332, USA.
E-mail: {\tt warnke@math.gatech.edu}. Research partially supported by NSF Grant DMS-1703516 and a Sloan Research Fellowship.}}
\date{May~21, 2019; revised June~10, 2019} 
\maketitle

\begin{abstract}
This note contains a short and simple proof of Wormald's differential equation method  
(that yields slightly improved approximation guarantees and error probabilities).
This powerful method uses differential equations to approximate 
the time-evolution/dynamics of random processes and algorithms. 
\end{abstract}

\section{Introduction}
Oftentimes it is natural and useful to approximate the trajectories of a random process
by the solutions to differential equations 
(whose deterministic behaviour is easier to understand). 
This widely-used approach has a long history: 
it was pioneered around~1970 by~Kurtz~\cite{Kurtz1970,Kurtz1981} for continuous-time Markov chains,
and introduced to the computer science community in the early~1980s by~Karp and Sipser~\cite{KS1981}.
In combinatorics this approach was popularized in the~1990s by~Wormald~\cite{DEM,DEM99,DEM2003,Seierstad2009,DM}: 
he developed a general framework for applying the so-called \emph{differential equation method} 
to discrete-time randomized algorithms and random combinatorial structures, 
which in the late~2000s has undergone some further technical advances 
by~Bohman and~others~\cite{B2009,BK2010,W2014,RW2016,BFL2015,FPGM2013,BK2013}. 
To date, this powerful method remains the state-of-the-art for the analysis of many important 
randomized combinatorial algorithms (see, e.g.,~\cite{B2009,BK2010,BFL2015,FPGM2013,BK2013,BW2018}).

In this note we provide a short and simple proof of Wormald's differential equation method~\cite{DEM,DEM99}, 
obtaining slightly improved approximation guarantees and error probabilities. 
The organization is as follows. 
In the next two subsections we illustrate the flavour of the method, and motivate the core proof idea in a simpler toy setting. 
In Section~\ref{sec:main} we then state and prove Wormald's theorem (Theorem~\ref{thm:DEM} and Section~\ref{sec:proof}), 
and also discuss some useful extensions (Section~\ref{sec:ext}). 
The final Section~\ref{sec:final} contains some brief concluding~remarks.

\subsection{Flavour of the method}\label{sec:flavour}
The basic goal of the differential equation method is to `track' a collection of 
variables~$(Y_k(i))_{1 \le k \le a}$ associated to some discrete-time random process 
(e.g., in some $n$-vertex random graph process, $Y_k(i)$~might 
denote the number of vertices of degree~$k$ after~$i$ steps),  
and it provides a framework for showing that these random 
variables closely `follow' the 
solution~$(y_k(t))_{1 \le k \le a}$ of a corresponding 
system of differential equations. 
The flavour of applications is roughly as follows: 
if the one-step changes of these variables~satisfy 
\vspace{-0.125em}\begin{itemize}
\itemsep 0.125em \parskip 0em  \partopsep=0pt \parsep 0em 
	\item $\E\bigpar{Y_k(i+1)-Y_k(i) \mid Y_1(i), \ldots, Y_a(i)} = F_k\bigpar{i/n,Y_1(i)/n,...,Y_a(i)/n} + o(1)$, where the functions~$F_k$ are `well-behaved' (i.e., sufficiently smooth), and 
	\item $\bigabs{Y_k(i+1)-Y_k(i)}$  is never `too big' (in the worst case), 
\end{itemize}\vspace{-0.125em}
then the heuristic conclusion of 
the differential equation method (see Theorem~\ref{thm:DEM}) is that 
\vspace{-0.125em}\begin{itemize}
\itemsep 0.125em \parskip 0em  \partopsep=0pt \parsep 0em 
	\item  with high probability~$Y_k(tn)=y_k(t)n+o(n)$, where the deterministic functions~$(y_k(t))_{1 \le k \le a}$ are the unique solution to~$y'_k(t) = F_k\bigpar{t,y_1(t),...,y_a(t)}$ with~$y_k(0) = Y_k(0)/n$.
\end{itemize}\vspace{-0.125em}
In concrete words,  
this says that 
if we interpret the expected one-step difference equations
 of the random variables as differential equations, 
then the values of the rescaled random variables~$Y_k(tn)/n$ typically stay close to the deterministic solutions~$y_k(t)$~of the 
corresponding system of differential equations 
(so the~$y_k(t)$ are the deterministic `limiting objects' of the~$Y_k(tn)/n$). 
This also establishes a form of \emph{dynamic concentration}, 
since the variables~$(Y_k(i))_{1 \le k \le a}$ are sharply concentrated around their expected trajectories in each~step.

\pagebreak[2]

\subsection{Motivation: stability of differential equations}\label{sec:WDEM:intuition}   
It turns out that the statement and proof of the differential equation method (Theorem~\ref{thm:DEM}) 
can be motivated by `stability properties' of differential equations with Lipschitz properties.  
The relevant toy question is: 
how much can two collections of functions~$(y_k(t))_{1 \le k \le a}$ and~$(z_k(t))_{1 \le k \le a}$ 
differ if they have similar derivatives and initial values?
To be more precise, assume that for some small `perturbations'~$\lambda,\delta$ we have 
\begin{align}
\label{eq:yk}
y_k(0)=\hat{y}_k & \quad \text{ and } \quad \; y'_k(t)=F_k\bigpar{t,y_1(t), \ldots, y_a(t)},\\
\label{eq:zk}
\abs{z_k(0)-\hat{y}_k} \le \lambda \:\: & \quad \text{ and } \quad \bigabs{z'_k(t)-F_k\bigpar{t,z_1(t), \ldots, z_a(t)}} \le \delta, 
\end{align}
where the functions~$F_k$ are $L$-Lipschitz-continuous\footnote{A function~$f$ is said to be $L$-Lipschitz-continuous on~$\cD \subseteq \RR^{\ell}$ if~$|f(x)-f(x')| \le L  \cdot \max_{1 \le k \le \ell}|x_k-x'_k|$ holds for all points~$x =(x_1, \ldots, x_\ell)$ and~$x'=(x'_1, \ldots, x'_\ell)$ in $\cD$, where~$\max_{1 \le k \le \ell}|x_k-x'_k|$ is the~$\ell^{\infty}$-distance between~$x$ and~$x'$.} 
on some bounded domain~$\cD \subseteq \RR^{a+1}$ (i.e., $\cD$ is a connected open subset that is bounded). 
In stability theory of differential equations it is standard to compare such functions 
via Gronwall's inequality (see Appendix~\ref{sec:apx} for its simple proof). 
\begin{lemma}[Gronwall's inequality]\label{lem:gronwall:int}
Given a continuous function~$x(t)$ defined on~$[0,T]$, assume that there is~$L \ge 0$ 
such that~$x(t) \le C + L \int_{0}^t x(s) \dd s$ for~$t \in [0,T)$. 
Then~$x(t) \le C e^{L t}$ for~$t \in [0,T]$.   
\end{lemma}
\noindent
Indeed, integrating the derivatives~$z'_k(t)$ and~$y'_k(t)$, after taking absolute values it follows that 
\begin{equation}\label{eq:error:int:0}
\begin{split}
|z_k(t)-y_k(t)| & \le |z_k(0)-y_k(0)| + \int_0^{t} |z'_k(s)-y'_k(s)| \dd s \\
& \le \lambda + \delta t +  \int_0^{t} \bigabs{F_k\bigpar{s,z_1(s), \ldots, z_a(s)}-F_k\bigpar{s,y_1(s), \ldots, y_a(s)}} \dd s . 
\end{split}
\end{equation}
Assuming for the moment that~$(t,z_1(t), \ldots, z_a(t))$ and~$(t,y_1(t), \ldots, y_a(t))$ 
are both still inside the domain~$\cD$ for all~$t \in [0,T)$, 
using that the functions~$F_k$ are $L$-Lipschitz-continuous on~$\cD$ it follows that 
\begin{equation}\label{eq:error:int:1}
\max_{1 \le k \le a} |z_k(t)-y_k(t)| \; \le \; (\lambda + \delta T) + L \int_0^{t} \max_{1 \le k \le a}|z_k(s)-y_k(s)| \dd s ,
\end{equation}
which by Gronwall's inequality (Lemma~\ref{lem:gronwall:int}) 
then implies for all~$t \in [0,T]$ the bound
\begin{equation}\label{eq:error:int}
\max_{1 \le k \le a} |z_k(t)-y_k(t)|  \le \big(\lambda + \delta T) e^{LT} .
\end{equation}
The punchline of the above argument is as follows: 
in order to understand the `approximate' solutions~$z_k(t)$ to the differential equation~\eqref{eq:yk}, 
it essentially suffices to understand the `exact' solutions~$y_k(t)$ on the domain~$\cD$.  
One snag is that, due to the error-term in~\eqref{eq:error:int}, 
it can happen that~$(t,z_1(t), \ldots, z_a(t)) \notin\cD$ despite $(t,y_1(t), \ldots, y_a(t)) \in \cD$. 
To overcome this obstacle, we intuitively remove all points from~$\cD$ which are `too close' to the boundary, 
so that~\eqref{eq:error:int} ensures~$(t, z_1(t), \ldots, z_a(t)) \in \cD$ in the remainder. 
The crux of this technical idea is that, by choosing~$\sigma \in [0,T]$ suitably, we can 
then ensure that the bound~\eqref{eq:error:int} holds for all~$t \in [0,\sigma]$.

Perhaps surprisingly, our upcoming proof of the differential equation method 
does little more than adapting the above comparison argument to the random setting 
(using concentration inequalities and a discrete variant of Gronwall's inequality). 
We hope that this viewpoint also clarifies the role of the somewhat technical 
parameter~$\sigma$ in Theorem~\ref{thm:DEM} below, which simply handles 
complications 
near the boundary of the~domain~$\cD$.

\section{Wormald's theorem}\label{sec:main}
We now state a \emph{non-asymptotic} version of the differential equation method  
which (together with Remark~\ref{rem:DEM:Glambda} and Lemma~\ref{lem:ext:trunc}) 
is slightly stronger than Wormald's original formulation~\cite{DEM,DEM99}.

\enlargethispage{\baselineskip}

The key difference lies in the (exponentially small) probability with which the conclusion~\eqref{dem:error} fails: 
in~\cite[Theorem~5.1]{DEM99} it goes to zero when~$n \lambda^3/\beta^3 \to \infty$, 
whereas in the Theorem~\ref{thm:DEM} below the weaker assumption~$n \lambda^2/\beta^2 \to \infty$ suffices 
(usually~$\lambda = o(1)$ and~$\beta=\Omega(1)$ hold). 
Besides better probability bounds, this also enables smaller `approximation errors' of the form~$O(\lambda n)$ in the conclusion~\eqref{dem:error} below.  
(For example, if we are as in~\cite[Section~4.2]{RWapbsr} aiming at error probabilities of form~$n^{-\omega(1)}$ when~$\beta=\Theta(1)$ and~$\delta=O(n^{-1})$, 
then~\cite[Theorem~5.1]{DEM99} requires~$\lambda n = \omega(1) \cdot n^{2/3}\sqrt[3]{\log n}$, 
whereas~$\lambda n = \omega(1) \cdot n^{1/2} \sqrt{\log n}$ suffices in Theorem~\ref{thm:DEM} below.)

In applications usually~$(\cF_i)_{i \ge 0}$ denotes the natural filtration of the underlying random~process, 
and then one can simply think of~$\cF_i$ as the~`history' which contains all information available during the first~$i$~steps. 
\pagebreak[2]
\begin{theorem}[Differential equation method]\label{thm:DEM}%
Given integers~$a,n \ge 1$, 
a bounded domain~$\cD \subseteq \RR^{a+1}$, 
functions~$(F_k)_{1 \le k \le a}$ with~$F_k:\cD \to \RR$, 
and $\sigma$-fields~$\cF_0 \subseteq \cF_1 \subseteq \cdots$, 
suppose that the random variables~$((Y_k(i))_{1 \le k \le a}$ are~$\cF_i$-measurable for~$i \ge 0$.   
Furthermore, assume that, for all~$i \ge 0$ and~$1 \le k \le a$, the following 
conditions hold whenever~$(i/n,Y_1(i)/n,...,Y_a(i)/n) \in \cD$: 
\vspace{-0.25em}%
\begin{enumerate}%
\itemsep 0.125em \parskip 0em  \partopsep=0pt \parsep 0em 
\item[(i)]\label{dem:trend}%
 $\bigabs{\E\bigpar{Y_k(i+1)-Y_k(i) \mid \cF_{i}}-F_k\bigpar{i/n,Y_1(i)/n,...,Y_a(i)/n}} \le \delta$, 
where the function~$F_k$ is~$L$-Lipschitz-continuous on~$\cD$ \ \emph{(the `Trend hypothesis' and `Lipschitz hypothesis')}, 
\item[(ii)]\label{dem:bounded}%
$\bigabs{Y_k(i+1)-Y_k(i)}\le \beta$ \ \emph{(the `Boundedness hypothesis')}, 
\end{enumerate}\vspace{-0.125em}%
and that the following condition holds initially: 
\vspace{-0.25em}%
\begin{enumerate}%
\itemsep 0.125em \parskip 0em  \partopsep=0pt \parsep 0em 
\item[(iii)]\label{dem:init}%
$\max_{1 \le k \le a} \bigabs{Y_k(0)- \hat{y}_k n} \le \lambda n$ for some~$(0,\hat{y}_1, \ldots, \hat{y}_a) \in \cD$ 
 \ \emph{(the `Initial condition')}. 
\end{enumerate}\vspace{-0.125em}%
Then there are~$R=R(\cD,(F_k)_{1 \le k \le a},L) \in [1,\infty)$ and~$T=T(\cD) \in (0,\infty)$ such that, 
whenever~$\lambda \ge \delta \min\{T,L^{-1}\} + R/n$,  
with probability at least $1-2a e^{-n\lambda^2/(8T \beta^2)}$ we have 
\begin{equation}\label{dem:error}
\max_{0 \le i \le \sigma n} \max_{1 \le k \le a}\bigabs{Y_k(i)-y_k\bigpar{\tfrac{i}{n}}n} \; < \; 3 e^{L T} \lambda n ,
\end{equation}
where~$(y_k(t))_{1 \le k \le a}$ is the unique solution to the system of differential equations
\begin{equation}\label{dem:sol}
y'_k(t) =F_k\bigpar{t,y_1(t), \ldots, y_a(t)} \quad \text{ with } \quad y_k(0) = \hat{y}_k \qquad \text{for~$1 \le k \le a$,}
\end{equation}
and~$\sigma=\sigma(\hat{y}_1, \ldots \hat{y}_a) \in [0,T]$ is any choice of~$\sigma \ge 0$ with the property\footnote{It suffices if~$\max_{1 \le k \le a}|\ty_k-y_k(t)| < 3 e^{LT}\lambda$ implies~$(t,\ty_1, \ldots, \ty_a) \in \cD$ for all~$t \in [0,\sigma)$ and~$(\ty_1, \ldots, \ty_a) \in \RR^a$.}  
that~$(t,y_1(t), \ldots y_a(t))$ has~$\ell^{\infty}$-distance at least~$3 e^{L T} \lambda$ 
from the boundary~of~$\cD$ for all~$t \in [0,\sigma)$.  
\end{theorem}
\begin{remark}\label{rem:DEM:Glambda}%
The deterministic `Initial condition'~(iii) can be~relaxed: 
the proof shows~$\Pr(\neg\cG_\lambda) \le 2a \cdot e^{-n\lambda^2/(8T\beta^2)}$, 
where~$\cG_\lambda$ is the event that~\eqref{dem:error} holds 
for \emph{all}~$(0,\hat{y}_1, \ldots, \hat{y}_a) \in \cD$ 
with~$\max_{1 \le k \le a} |Y_k(0)- \hat{y}_k n| \le \lambda n$. 
\end{remark}
\noindent 
The surveys~\cite{DEM99,DM} contain numerous examples that 
illustrate how to apply this powerful result 
(some technical extensions of Theorem~\ref{thm:DEM} 
are discussed in Section~\ref{sec:ext}, which, e.g., allow for larger one-step changes). 
We point out that~\eqref{dem:error} only gives a `good' approximation 
as long as~$|y_k(i/n)| \gg 3 e^{L T} \lambda$, 
which for many natural choices of~$\cD$ and~$\sigma$ 
means that the condition~$i/n \leq \sigma$ in~\eqref{dem:error} is not very restrictive;  
see also~Section~\ref{sec:ext:sigma}. 
\begin{remark}\label{rem:DEM:unique}%
Standard results for differential equations (see, e.g.,~\cite[Theorem~11 in Chapter~2.5]{Hurewicz}) 
guarantee that~\eqref{dem:sol} has a unique solution~$(y_k(t))_{1 \le k \le a}$ 
which extends arbitrarily close to the boundary of~$\cD$. 
\end{remark}
\begin{remark}\label{rem:DEM:RT}%
The proof of Theorem~\ref{thm:DEM} in fact works for any choice of~$R \in [1,\infty)$ and~$T \in (0,\infty)$ 
which satisfy~$t \le T$ and $\max_{1 \le k \le a}|F_k(x)| \le R$ for all~$x = (t,y_1, \ldots, y_a) \in \cD$ with~$t \ge 0$. 
\end{remark}


\subsection{Proof}\label{sec:proof} 
The below proof of Theorem~\ref{thm:DEM} essentially mimics the deterministic Gronwall-type argument 
from Section~\ref{sec:WDEM:intuition} in the present random setting 
(this strategy differs slightly from the original proof given by Wormald~\cite{DEM,DEM99},  
and resembles more some earlier arguments of Kurtz~\cite{Kurtz1970,Kurtz1981} for continuous-time Markov chains). 
\begin{proof}[Why is Theorem~\ref{thm:DEM} intuitively true?]%
First we reduce to a deterministic setting: 
combining the Azuma--Hoeffding inequality (Lemma~\ref{lem:AH} below) 
with the `Boundedness hypothesis'~(ii) and the `Trend hypothesis'~(i), it turns out that, 
with sufficiently high probability, for all relevant~$j$ 
the random variables~$Y_k(j)$ approximately satisfy 
\begin{equation}\label{eq:heur:Yk:decomp}
Y_k(j) - Y_k(0) 
\approx  \sum_{0 \le i < j} \E (Y_k(i+1)-Y_k(i) \mid \cF_i) 
= \sum_{0 \le i < j} \Bigsqpar{F_k\Bigpar{\tfrac{i}{n}, \tfrac{Y_1(i)}{n}, \ldots, \tfrac{Y_k(i)}{n}} \pm \delta}.
\end{equation}
Second, the solutions~$y_k(t)$ to the system of differential equations~\eqref{dem:sol} approximately satisfy 
\begin{equation}\label{eq:heur:yk:decomp}
y_k(\tfrac{j}{n})n - y_k(0)n 
=\sum_{0 \le i < j}\underbrace{\bigsqpar{y_k\bigpar{\tfrac{i+1}{n}}-y_k\bigpar{\tfrac{i}{n}}}n}_{\approx y'_k\bigpar{\tfrac{i}{n}}} 
\approx 
\sum_{0 \le i < j} F_k\Bigpar{\tfrac{i}{n}, y_1\bigpar{\tfrac{i}{n}}, \ldots, y_a\bigpar{\tfrac{i}{n}}} .
\end{equation}
Comparing these two expressions analogous to~\eqref{eq:error:int:0}--\eqref{eq:error:int}, 
using that the functions~$F_k$ are $L$-Lipschitz-continuous 
we then bound~$\max_{1 \le k \le a}|Y_k(i)-y_k(\tfrac{i}{n})n|$ 
via a discrete variant of Gronwall's inequality (Lemma~\ref{lem:gronwall} below). 
This intuitively yields the main estimate~\eqref{dem:error} by mimicking the 
arguments from~Section~\ref{sec:WDEM:intuition}
(taking into account~$|Y_k(0)-y_k(0)n| \le \lambda n$, 
the approximation errors in~\eqref{eq:heur:Yk:decomp}--\eqref{eq:heur:yk:decomp}, 
and domain boundary complications). 
\end{proof}
\noindent 
The formal details rely on the following standard inequalities 
(see Appendix~\ref{sec:apx} for simple~proofs). 
\begin{lemma}[Discrete Gronwall's inequality]\label{lem:gronwall}%
Assume that there are~$b,c \ge 0$ and~$a > 0$ such that~$x_j < c + \sum_{0 \le i < j} (a x_i+b)$ for all~$0 \le j \le m$. 
Then~$x_m < \bigpar{c+b\min\{m,a^{-1}\}} e^{am}$. 
\end{lemma}
\begin{lemma}[Azuma--Hoeffding inequality]\label{lem:AH}%
Assume that~$(M_i)_{0 \le i \le m}$ are~$\cF_{i}$-measurable random variables 
satisfying~$M_0=0$ as well as~$\E(M_{i+1}-M_{i} \mid  \cF_{i}) = 0$ and~$|M_{i+1}-M_{i}| \le c$ for~$0 \le i < m$.
Then~$\Pr(\max_{0 \le j \le m}|M_j| \ge t) \le 2e^{-t^2/(2 m c^2)}$ for all~$t \ge 0$.  
\end{lemma}
\begin{proof}[Proof of Theorem~\ref{thm:DEM}]%
We start with the \emph{analytic part of the argument}. 
Since the functions~$F_k$ are~$L$-Lipschitz on the bounded domain~$\cD$, 
it routinely follows 
that there exist~$T \in (0,\infty)$ and~$R\in [1,\infty)$ satisfying~$t \le T$ and~$\max_{1 \le k \le a}|F_k(x)| \le R$ for all~$x=(t,y_1, \ldots, y_a) \in \cD$. 
By~\eqref{dem:sol} and Remark~\ref{rem:DEM:unique}, 
in the upcoming arguments we thus (for all~$(0,\hat{y}_1, \ldots, \hat{y}_a) \in \cD$) 
always have~$\sigma \in [0,T]$, and~$\max_{1 \le k \le a}|y'_k(t)| \le R$ for all~$t \in [0,\sigma]$.

The next \emph{probabilistic part of the argument} is based on the Azuma--Hoeffding inequality 
(and a Doob decomposition). 
Define~$I_{\cD}$ as the minimum of~$\floor{Tn}$ and the smallest integer~$i \ge 0$ where~$(i/n,Y_1(i)/n,...,Y_a(i)/n) \not\in \cD$ holds. 
Set~$\Delta Y_{k}(i) := \indic{i < I_{\cD}}[Y_{k}(i+1)-Y_{k}(i)]$ 
and~$M_{k}(j) := \sum_{0 \le i < j} \bigsqpar{\Delta Y_{k}(i) - \E(\Delta Y_{k}(i) \mid \cF_{i})}$. 
Since the event~$\{i < I_{\cD}\}$ is $\cF_{i}$-measurable 
(determined by all information of the first~$i$ steps),
we have 
\begin{equation}\label{eq:MD}
Y_{k}(j) = M_{k}(j) + Y_{k}(0) + \sum_{0 \le i < j} \E(Y_{k}(i+1)-Y_{k}(i)  \mid \cF_{i}) \qquad \text{for all~$0 \le j \le I_{\cD}$.}
\end{equation}
Furthermore, for all $i \ge 0$, 
the `tower property' of conditional expectations 
implies~$\E(M_{k}(i+1)-M_{k}(i) \mid \cF_{i}) = \allowbreak \indic{i < I_{\cD}}\E\bigpar{Y_{k}(i+1)-Y_{k}(i) - \E(Y_{k}(i+1)-Y_{k}(i) \mid \cF_{i})\big| \cF_{i}} \allowbreak = 0$, 
and the `Boundedness hypothesis'~(ii) implies~$|M_{k}(i+1)-M_{k}(i)| \le 2\beta$. 
Defining~$\cM$ as the event that~$\max_{0 \le j \le I_{\cD}}|M_{k}(j)| < \lambda n$ for all~$1 \le k \le a$, 
the Azuma--Hoeffding inequality (Lemma~\ref{lem:AH} with~$m:=\floor{Tn}$) thus 
yields~$\Pr(\neg \cM) \le a \cdot 2e^{-n\lambda^2/(8T\beta^2)}$.


The final \emph{deterministic part of the argument} is based on a discrete variant of Gronwall's inequality (and induction). 
Assuming that the event~$\cM$ holds, for all~$(0,\hat{y}_1, \ldots, \hat{y}_a) \in \cD$ 
satisfying~$\max_{1 \le k \le a} |Y_k(0)- \hat{y}_k n| \le \lambda n$ (see Remark~\ref{rem:DEM:Glambda}) 
it remains to prove by induction that, for all integers~$0 \le m \le \sigma n$, we~have 
\begin{equation}\label{eq:DEM:error}
\max_{1 \le k \le a} \bigabs{Y_{k}(m)- y_k\bigpar{\tfrac{m}{n}}n} \; < \; 3 \lambda n e^{LT}.
\end{equation}
The base case~$m=0$ holds since~$\max_{1 \le k \le a} |Y_k(0)- \hat{y}_k n| \le \lambda n$ by assumption and~$\hat{y}_k=y_k(0)$ by~\eqref{dem:sol}.
Turning to the induction step~$1 \le m \le \sigma n$, note that~$m-1 < \floor{\sigma n} \le \floor{Tn}$ by the analytic part. 
So, by choice of~$\sigma$ the induction hypothesis implies~$m-1 < I_{\cD}$ and thus~$m \le I_{\cD}$. 
Fix~$0 \le j \le m$. Writing~$y_k(j/n)$ as the sum~of consecutive differences as in~\eqref{eq:heur:yk:decomp},
by combining~\eqref{eq:MD} and the event~$\cM$  %
with~$|Y_k(0)- y_k(0) n| \le \lambda n$~and 
the~`Trend hypothesis'~(i), it then follows 
(using~$j \le m \le I_{\cD}$, $(i+1)/n \le m/n \le \sigma$, and the mean value~theorem)~that 
\begin{align}
\nonumber
\bigabs{Y_{k}(j)- y_k\bigpar{\tfrac{j}{n}}n} & \le |M_{k}(j)| +  |Y_k(0)- y_k(0) n|  + \sum_{0 \le i < j} \Bigabs{\E\bigpar{Y_{k}(i+1)-Y_{k}(i)  \mid \cF_{i}} -\bigsqpar{y_k\bigpar{\tfrac{i+1}{n}}-y_k\bigpar{\tfrac{i}{n}}}n} \\
\label{eq:dem:rec}
& < 2\lambda n + \sum_{0 \le i < j} 
\biggpar{\sup_{\xi \in (\frac{i}{n},\frac{i+1}{n})}\underbrace{\Bigabs{F_k\bigpar{\tfrac{i}{n},\tfrac{Y_1(i)}{n},\ldots,\tfrac{Y_k(i)}{n}} - y'_k(\xi)}}_{=: \Lambda_k(i,\xi)}+\delta} .
\end{align}
Let~$\oY(i)/n:=(Y_1(i)/n, \ldots, Y_a(i)/n)$ and~$\oy(t):=(y_1(t), \ldots, y_a(t))$. 
Since~$y_k'(\xi)=F_k(\xi,\oy(\xi))$ by~\eqref{dem:sol},   
and~$F_k$ is~$L$-Lipschitz on~$\cD$, 
using~$i/n < \xi < \sigma$ and~$|i/n-\xi| \le 1/n$ it follows that~$\Lambda_k(i,\xi)$ from~\eqref{eq:dem:rec}~satisfies 
\begin{equation*}
\begin{split}
\Lambda_k(i,\xi) & \le \Bigabs{F_k\bigpar{\tfrac{i}{n},\tfrac{\oY(i)}{n}} - F_k\bigpar{\tfrac{i}{n},\oy\bigpar{\tfrac{i}{n}}}} + \Bigabs{F_k\bigpar{\tfrac{i}{n},\oy\bigpar{\tfrac{i}{n}}}-F_k\bigpar{\xi,\oy\bigpar{\xi}}}\\
& \le L \cdot \max_{1 \le k \le a} |Y_k(i)/n- y_k(i/n)| + L \cdot \max\Bigset{\tfrac{1}{n}, \; \max_{1 \le k \le a}\bigabs{y_k\bigpar{\tfrac{i}{n}}-y_k(\xi)}} .
\end{split}
\end{equation*}
Since~$|y_k(i/n)-y_k(\xi)| \le \int_{i/n}^{\xi}|y'_k(t)|\dd t \le R/n$, 
in view of~$R \ge 1$ it follows, for all~$0 \le j \le m$, that 
\begin{equation}\label{eq:dem:GW}
\max_{1 \le k \le a} \bigabs{Y_{k}(j)- y_k\bigpar{\tfrac{j}{n}}n} \; < \; 2\lambda n + \sum_{0 \le i < j} \Bigsqpar{\tfrac{L}{n} \cdot \max_{1 \le k \le a} \bigabs{Y_{k}(i)- y_k\bigpar{\tfrac{i}{n}}n} + \Bigpar{\tfrac{L}{n}\cdot R +\delta}} . 
\end{equation} 
Recalling~$m \le \sigma n \le Tn$, the discrete Gronwall's inequality (Lemma~\ref{lem:gronwall}) thus yields
\begin{equation}\label{eq:dem:GW:ineq}
\max_{1 \le k \le a} \bigabs{Y_{k}(m)- y_k\bigpar{\tfrac{m}{n}}n} < \Bigpar{2\lambda n  + \bigsqpar{R + \delta \min\{Tn,n/L\}}} e^{Lm/n} \le 3 \lambda n e^{LT}
\end{equation}
by the assumed lower bound on~$\lambda$, completing the inductive proof of inequality~\eqref{eq:DEM:error}.
\end{proof}

\subsection{Useful extensions and ideas}\label{sec:ext}
One benefit of the simple proof of Theorem~\ref{thm:DEM} 
(with separate analytic, probabilistic and deterministic parts)
is that it can be easily modified. 
We shall illustrate this using a few concrete extensions, 
where the ideas from Lemmas~\ref{lem:ext:side}--\ref{lem:ext:conc} and Example~\ref{ex:domain} 
seem particularly useful (in addition to the event~$\cG_{\lambda}$ from~Remark~\ref{rem:DEM:Glambda}).

\subsubsection{Weakening the conditions}\label{sec:ext:conditions}
As a first example, by inspecting the proof itself or using additional typical properties, 
it suffices to verify conditions~(i) and~(ii) of Theorem~\ref{thm:DEM} under 
helpful  extra assumptions (which can often lead to significantly improved bounds on the parameters~$\delta,\beta$, 
in particular via Lemma~\ref{lem:ext:side} below). 
\begin{lemma}[Exploiting the proof structure]\label{lem:ext:proof}%
In the verification of conditions~(i) and~(ii) we may additionally assume 
that~$i < \floor{Tn}$ and~$\max_{1 \le k \le a} \bigabs{Y_{k}(i)- y_k\bigpar{\tfrac{i}{n}}n} < 3e^{LT}\lambda n$ hold.  
\end{lemma}
\begin{proof}%
In the probabilistic~part of the argument we redefine~$I_{\cD}$ 
as the minimum of~$\floor{Tn}$ and the smallest integer~$i \ge 0$ where~$(i/n,Y_1(i)/n,...,Y_a(i)/n) \not\in \cD$ 
or~$\max_{1 \le k \le a} \bigabs{Y_{k}(i)- y_k\bigpar{\tfrac{i}{n}}n} \ge 3e^{LT}\lambda n$ holds. 
The deterministic~part 
then carries over, since 
the induction hypothesis again ensures~$m-1 < I_{\cD} \le \floor{Tn}$. 
\end{proof}
\begin{lemma}[Using additional events]\label{lem:ext:side}%
Given~$\cF_i$--measurable events~$(\cE_i)_{0 \le i < I}$, assume that we only verify conditions~(i) and~(ii) under the additional assumption that~$\cE_i$ holds. 
After replacing~$0 \le i \le \sigma n$ in~\eqref{dem:error} with~$0 \le i \le \min\{\sigma n,I\}$, 
then~$\Pr(\neg\cG_\lambda \cap \cE) \le 2a \cdot e^{-n\lambda^2/(8T \beta^2)}$ holds for any event~$\cE \subseteq \bigcap_{0 \le i < I} \cE_i$.   
\end{lemma}
\begin{proof}%
In the probabilistic part of the argument
 we redefine~$\Delta Y_{k}(i) := \indic{i < I_{\cD} \text{ and } \cE_i}[Y_{k}(i+1)-Y_{k}(i)]$. 
If the event~$\cE$~holds, then the crux is that for all $0 \le j \le \min\{I_{\cD}, I\}$ the key equation~\eqref{eq:MD} remains valid, 
so the rest of the proof carries over essentially unchanged
(assuming that~$\cM \cap \cE$ holds in the deterministic~part). 
\end{proof}

\subsubsection{Dealing with large one-step changes}\label{sec:ext:changes} 
As a second example, by using average one-step bounds or truncation arguments,
we can often handle (via refined error probabilities) much 
larger 
one-step changes in the `Boundedness hypothesis'~(ii) 
of~Theorem~\ref{thm:DEM}.   
\begin{lemma}[Using average one-step bounds]\label{lem:ext:conc}%
Assume that we add the bound~$\E\bigpar{|Y_k(i+1)-Y_k(i)| \mid \cF_i} \le b$ to condition~(ii). 
Then~$\Pr(\neg\cG_\lambda) \le 2a \cdot e^{-\min\{n\lambda^2/(4T \beta b), \: n\lambda/(4\beta)\}}$. 
\end{lemma}
%
%
\begin{proof}%
We shall only modify the probabilistic part of the argument, replacing  
the Azuma--Hoeffding inequality by a more advanced (martingale) concentration inequality that can be traced back to Freedman~\cite{F75}. 
Using standard variance properties, 
the crux is that the modified `Boundedness hypothesis'~(ii) implies 
\[
\Var(M_k(i+1)-M_k(i) \mid \cF_i) = \Var(\Delta Y_k(i) \mid \cF_i) 
\le \indic{i < I_{\cD}} \E\bigpar{|Y_k(i+1)-Y_k(i)|^2 \mid \cF_i} \le \beta \cdot b.
\]
Recalling~$M_k(0)=0$, $\E(M_k(i+1)-M_k(i) \mid \cF_i)=0$ and~$|M_k(i+1)-M_k(i)| \le 2\beta$, 
now~\cite[Lemma~2.2]{TBD} routinely yields~$\Pr(\neg\cM) \le a \cdot 2 e^{-(\lambda n)^2/(2Tn \beta b +2\beta\lambda n)}$, 
and the rest of the proof carries~over. 
\end{proof}
\begin{lemma}[Truncating large one-step changes]\label{lem:ext:trunc}%
Assume that we replace condition~(ii) by~$\Pr\bigpar{|Y_k(i+1)-Y_k(i)|> \beta \mid \cF_i} \allowbreak \le \gamma$ and~$\bigabs{Y_k(i+1)-Y_k(i)}\le B$.  
Then, for all~$x \ge 0$ and~$\lambda \ge (\delta+\gamma B) \min\{T,L^{-1}\} \allowbreak + (R+xB)/n$, 
we have~$\Pr(\neg\cG_\lambda) \le 2a \cdot e^{-n\lambda^2/(8T \beta^2)} + a  \cdot \Pr(Z \ge \floor{x+1})$ with~$Z \sim \Bin(\floor{Tn},\gamma)$.  
\end{lemma}
\begin{remark}%
Note that~$\Pr(Z \ge 1) \le Tn \gamma$, and~$\Pr(Z \ge \floor{x+1}) \le (eTn\gamma/\ceil{x})^{\ceil{x}}$ for~$x>0$. 
\end{remark}
\begin{proof}%
Writing~$Z_k$ for the number of~$0 \le i < I_{\cD}$ with~$|Y_k(i+1)-Y_k(i)| > \beta$, 
using the modified `Boundedness hypothesis'~(ii) it easily follows that~$Z_k$ is stochastically dominated by~$Z$. 
Defining $\cB_x := \allowbreak \{\max_{1 \le k \le a} Z'_k \le x\}$, 
we infer~$\Pr(\neg\cB_x) \le a \cdot \Pr(Z \ge \floor{x+1})$. 
In the probabilistic part of the argument 
we redefine~$\Delta Y_{k}(i) := \allowbreak \indic{i < I_{\cD} \text{ and } |Y_k(i+1)-Y_k(i)| \le \beta}[Y_{k}(i+1)-Y_{k}(i)]$. 
If~$\cB_x$ holds, 
then it is not difficult to see that for all~$0 \le j \le I_{\cD}$ 
the first key equation~\eqref{eq:MD} remains valid with 
(a)~$Y_{k}(0)$~replaced by~$Y_{k}(0) \pm xB$, 
and (b)~$\E(Y_{k}(i+1)-Y_{k}(i)  \mid \cF_{i})$ replaced by~$\E(Y_{k}(i+1)-Y_{k}(i)  \mid \cF_{i}) \pm \gamma B$; the rest of the probabilistic part routinely carries over. 
In the deterministic part of the argument (where we assume that~$\cM \cap \cB_x$ holds)  
now the second key equation~\eqref{eq:dem:GW} remains valid with 
(a)~$2 \lambda n$~replaced by~$2 \lambda n+xB$ and (b)~$\delta$~replaced by~$\delta+\gamma B$, 
and the increased lower bound on~$\lambda$ then allows us to absorb the corresponding 
extra terms in~\eqref{eq:dem:GW:ineq} into the final upper bound~$3\lambda n e^{LT}$. 
\end{proof}

\subsubsection{Choice of the parameter~$\sigma$}\label{sec:ext:sigma} 
Finally, to understand the somewhat technical parameter~$\sigma$ from Theorem~\ref{thm:DEM}, 
one usually needs to determine which domain constraint is `nearly violated' 
when the solution comes `too close' to the boundary of the domain~$\cD$. 
In many routine applications a natural choice of~$\cD$ can ensure that 
no function~$y_k(t)$ comes close to the boundary, 
in which case~$\sigma$ is solely determined by the time-variable~$t$. 
In the verification of such claims, 
instead of arguing directly about the solutions~$y_k(t)$ to the differential equations,  
here it often is convenient to `turn the differential equation method on its head' 
and use properties of the random variables~$Y_k(i)$ to derive properties of the deterministic functions~$y_k(t)$. 
We now formalize and illustrate this useful idea.%
%
\begin{lemma}[Relating random variables and differential equations]\label{lem:ext:connection}%
Assume that~$\lambda=o(1)$ as~$n \to \infty$. 
If there are constants~$A_k,B_k$ such that~$Y_k(i)/n \in [A_k,B_k]$ holds for all~$0 \le i \le \sigma n$, 
then the conclusion~\eqref{dem:error} of the differential equation method implies~$y_k(t) \in [A_k,B_k]$ for all~$t \in [0,\sigma]$. 
\end{lemma}
\begin{proof}
Note that~\eqref{dem:error} implies~$y_k(i/n) \in (A_k - 3e^{LT}\lambda, B_k + 3e^{LT}\lambda)$ for all~$0 \le i \le \sigma n$. 
Combining continuity of the (differentiable) functions~$y_k$ with~$1/n=o(1)$ and~$\lambda=o(1)$,  
for any~$\eps>0$ it follows that~$y_k(t) \in (A_k-\eps,B_k+\eps)$ for all~$t \in [0,\sigma]$, completing the proof. 
\end{proof}
\begin{example}[Choice of~$\sigma$]\label{ex:domain} 
Suppose that there are constants~$A_k,B_k$ such that~$Y_k(i)/n \in [A_k,B_k]$ holds for all~$0 \le i \le \sigma n$. 
Fixing~$\sigma \ge 0$ and small~$\eps>0$, suppose that the~$F_k$ are $L$-Lipschitz 
on the domain~$\cD=\cD_{\eps}$ which contains all~$(t, y_1, \ldots, y_a) \in \RR^{a+1}$ 
satisfying $t \in (-\eps,\sigma+\eps)$ and $y_k \in (A_k-\eps,B_k+\eps)$. 
The natural assumption~$\lambda=o(1)$ as~$n \to \infty$ then ensures via Lemma~\ref{lem:ext:connection} that 
the conclusion~\eqref{dem:error} of the differential equation method 
implies~$y_k(t) \in [A_k,B_k]$ for all~$t \in [0,\sigma]$.
Hence no~$y_k(t)$ can come~$3e^{LT}\lambda = o(\eps)$ close to the boundary for~$t \in [0,\sigma]$, 
which shows that~$\sigma$ is a valid choice in Theorem~\ref{thm:DEM} 
(since~$\sigma+3e^{LT}\lambda < \sigma + \eps/2$, say). 
\end{example}
\begin{remark}[Using additional events]\label{rem:ext:side}%
In the setting of Lemma~\ref{lem:ext:side} 
one can of course again argue about the parameter~$\sigma$ and the range of the functions~$y_k(t)$ as in Example~\ref{ex:domain} and Lemma~\ref{lem:ext:connection} above, 
provided that the additional event~$\cE$ implies the relevant bounds~$Y_k(i)/n \in [A_k,B_k]$ for all~$0 \le i \le \sigma n$, say. 
\end{remark}

\section{Concluding remarks}\label{sec:final}
In this note we have given a conceptually simple proof of 
Wormald's differential equation method that 
might be suitable for teaching in class 
(we tried to keep the entry-level low by avoiding `martingale~jargon'). 
Our slightly stronger conclusion is also useful for applications requiring small approximation errors, see~\cite{RWapbsr}.

We believe that the differential equations perspective taken in Section~\ref{sec:WDEM:intuition} 
facilitates the development of new proof approaches. 
Indeed, inequalities developed in that deterministic toy setting can sometimes be lifted to the 
random setting by adding martingale error terms to the argument, as exemplified 
by Section~\ref{sec:proof} (for~Gronwall's inequality).  
As a further illustration, suppose that we replace the second inequality of~\eqref{eq:zk} 
by the following stronger approximation assumption: 
if~$|z_k(t)-y_k(t)| \le \xi_k(t)$ holds for all $1 \le k \le a$, then 
\begin{align}
\label{eq:zk:better}
\bigabs{z'_k(t)-F_k\bigpar{t,y_1(t), \ldots, y_a(t)}} \le \delta_k(t) .
\end{align}
If, in addition, the integral inequality~$\xi_k(t) \ge \int_0^t \delta_k(s) \dd s + \lambda$ 
holds (a kind of `consistency equation' for the error terms), 
then a comparison argument along the lines of~\eqref{eq:error:int:0}--\eqref{eq:error:int} 
yields the bound 
\[
|z_k(t)-y_k(t)| \le \lambda + \int_0^t \delta_k(s) \dd s \le \xi_k(t) .
\]
Interestingly, it turns out that an adaptation of this idea to the 
discrete random setting naturally leads to an approach that 
is more or less equivalent to the one developed by Bohman~\cite{B2009,BK2010,W2014}.

\bigskip{\noindent\bf Acknowledgements.}  
The proof contained in this note (and its motivation) was presented 
in graduate courses at Cambridge~(2013), Georgia~Tech~(2017), 
and a Fields~Institute summer~school~(2017). 
I~thank Patrick~Bennett, Tom~Bohman, Catherine~Greenhill, 
Tamas~Makai, Oliver~Riordan, Greg~Sorkin, Joel~Spencer, 
and Nick~Wormald for helpful comments and discussions.

\small

\bibliographystyle{plain}

\normalsize

\vspace{-0.5em}
\enlargethispage{\baselineskip}

\begin{appendix}
\section{Appendix: deferred routine proofs}\label{sec:apx}
%
%
\begin{proof}[Proof of Lemma~\ref{lem:gronwall:int}]%
Let~$y(t) := L\int_{0}^{t} x(s) \dd s$. 
Noting~$(y(t) e^{-Lt})' = [y'(t)-Ly(t)]e^{-Lt} = L[x(t)-y(t)]e^{-Lt} \le LC e^{-Lt}$, 
integration gives~$y(t) e^{-Lt} \le -Ce^{-Lt}+C$ and~$x(t) \le C + y(t) \le C e^{Lt}$. 
\end{proof}
\begin{proof}[Proof of Lemma~\ref{lem:gronwall}]%
Assuming~$b=0$, for~$0 \le j \le m$ it is routine to verify by induction that   
\begin{equation*}
x_j < c + \sum_{0 \le i < j} a x_i \le c \Bigl(1+\sum_{0 \le i < j} a (1+a)^i\Bigr) 
= c (1+a)^j,
\end{equation*}
which together with~$1+a \le e^a$ yields~$x_m < c e^{a m}$. 
This special case implies the claimed 
bound, since~$(x_j+b/a) < (c + b/a) + \sum_{0 \le i < j} a (x_i+b/a)$ 
and~$x_j < (c + bm) + \sum_{0 \le i < j} a x_i$ for all~$0 \le j \le m$.
\end{proof}
\begin{proof}[Proof of Lemma~\ref{lem:AH}]%
Define~$\cE_i := \set{\max_{0 \le j \le i}M_j \ge t}$ and~$\lambda := t/(m c^2)$. 
Setting~$S_0:=M_0=0$, define~$S_i := \indic{\neg \cE_{i-1}}M_i + \indic{\cE_{i-1}}S_{i-1}$ 
and~$X_i:=S_{i+1}-S_i = \indic{\neg\cE_{i}}[M_{i+1}-M_{i}]$. 
Since~$x \mapsto e^{\lambda x}$ is a convex function 
and~$\tfrac{1}{2}e^{y}+\tfrac{1}{2}e^{-y} \le e^{y^2/2}$ holds 
(by comparing Taylor series), it follows that 
\begin{equation}\label{eq:AH:exp}
e^{\lambda x} \le \frac{x+c}{2c}e^{\lambda c} + \frac{c-x}{2c}e^{-\lambda c} \le \frac{x}{2c} \Bigpar{e^{\lambda c} - e^{-\lambda c}} + e^{(\lambda c)^2/2} \qquad \text{for~$-c \le x \le c$} .
\end{equation}
We have~$\E(X_i \mid \cF_{i}) = \indic{\neg\cE_{i}}\E(M_{i+1}-M_{i}\mid \cF_{i})=0$ and~$|X_i| \le |M_{i+1}-M_{i}| \le c$, 
so~$\E(e^{\lambda X_i} \mid \cF_{i}) \le e^{(\lambda c)^2/2}$ by~\eqref{eq:AH:exp}.
Iterating~$\E(e^{\lambda S_{i+1}}) = \E(e^{\lambda S_{i}} \E(e^{\lambda X_i} \mid \cF_{i})) \le \E(e^{\lambda S_{i}}) \cdot e^{(\lambda c)^2/2}$ yields~$\E(e^{\lambda S_m}) \le e^{m (\lambda c)^2/2}$. 
Noting that~$\cE_m$ implies~$S_m \ge t$, 
Markov's inequality and~$\lambda=t/(mc^2)$ thus~give 
\begin{equation}\label{eq:AH:pr}
\Pr\Bigpar{\max_{0 \le j \le m}M_j \ge t} = \Pr(\cE_m) \le \Pr(e^{\lambda S_m} \ge e^{\lambda t}) \le \E(e^{\lambda S_m}) e^{-\lambda t} \le e^{m(c\lambda)^2/2-\lambda t} = e^{-t^2/(2m c^2)} .
\end{equation}
Since~$\min_{0 \le j \le m}M_j \le -t$ implies~$\max_{0 \le j \le m}(-M_j) \ge t$, 
now a further application of~\eqref{eq:AH:pr} completes the proof 
(as~$M'_i :=-M_i$ also satisfies~$M'_0=0$, $\E(M'_{i+1}-M'_{i}\mid \cF_{i})=0$, and~$|M'_{i+1}-M'_{i}| \le c$). 
\end{proof}
\end{appendix}

\end{document}